\documentclass{amsart}

\usepackage{graphicx, amssymb}

\renewcommand{\epsilon}{\varepsilon}

\newcommand{\R}{\mathbb{R}}

\newcommand{\dd}{\,d}
\newcommand{\nl}{\left\lVert}
\newcommand{\nr}{\right\rVert}

\newcommand{\M}{\mathcal{M}}
\newcommand{\product}[1]{\mu_{\M}}

\renewcommand{\and}{\hspace{8pt} \text{and} \hspace{8pt}}

\DeclareMathOperator{\id}{id.}
\DeclareMathOperator{\im}{im}

\newtheorem{theorem}{Theorem}
\newtheorem{proposition}{Proposition}
\newtheorem{lemma}{Lemma}

\theoremstyle{definition}

\theoremstyle{remark}
\newtheorem*{remark}{Remark}

\begin{document}

\title[A Further Remark on Sobolev Spaces. The Case $0<p<1$]{A Further Remark on Sobolev Spaces. \\The Case $0<p<1$}
\author{Aron Wennman}
\address{Department of Mathematics, KTH Royal Institute for Technology, Stockholm}
\email{aronw@maths.kth.se}
\date{\today}


\begin{abstract}
We discuss a phenomenon observed by Jaak Peetre in the seventies: for small $L^{p}$-exponents, i.e.~for $0<p<1$, the Sobolev spaces $W^{k,p}$ defined in a seemingly natural way are isomorphic to $L^{p}$. 
This implies in particular that the dual of $W^{k,p}$ is trivial, and indicates that these spaces are highly pathological. 
In this note we expand on Peetre's observation,
explaining in detail some points that might merit further discussion.
\end{abstract}
\maketitle

\section{Introduction}

In \cite{Peetre}, Jaak Peetre studies naturally defined Sobolev spaces for small exponents, i.e.~$W^{k,p}$ for $0<p<1$. 
His main result is that these Sobolev spaces, defined as abstract completions of $C^{\infty}$, are far too large to be useful objects. 
More precisely the statement is that
$$
W^{k,p}\cong L^{p}\oplus L^{p}\oplus \ldots \oplus L^{p}\cong L^{p},
$$
where the isomorphisms involved are all topological. The last isomorphism is a well known fact, and by a classical theorem of Day \cite{Day} which characterizes the dual of $L^p$, it thus holds that the dual of $W^{k,p}$ is trivial. 
The proof of the main result in \cite{Peetre}, however, is based on a general algebraic lemma which is not entirely clear. 
For this reason we revisit the proof in the present note. The author wishes to point out that although he believes some of the ideas presented here to be his own, it is by no means a new proof; it is more of a re-examination of exposition and an additional lemma that seems to be needed in one form or another. 
It also turns out that one may fix many of the unclear points, especially in the $W^{1,p}$-case, by employing ideas due to Douady and due to Peetre, which are all discussed in \cite{Peetre}. 
The proof of surjectivity of the map $\delta$ defined below is probably the main contribution made here.
The contents of this note have are contained in the joint work \cite{HWLp} with Hedenmalm, concerning what we call $L^p$-Carleman classes.

We restate the result under consideration for easy reference.

\begin{theorem}[Peetre, \cite{Peetre}]\label{Peetre-main} We have $W^{k,p}\cong L^{p}\oplus W^{k-1, p}$.
\end{theorem}

The theorem will be proven using a sequence of lemmas in Section~\ref{correction}. First, we repeat some useful definitions.

\section{Definitions}
As alluded to above, we define $W^{k,p}$ for functions on an interval $I\subseteq\R$, which may very well be the entire line, as the abstract completion of functions in $C^{\infty}_{0}(\R)$ restricted to $I$, with respect to the Sobolev norm
\begin{equation}\label{sobolevnorm}
\nl f\nr_{W^{k,p}}=\left(\lVert f\rVert_{p}^{p}+\lVert f'\rVert_{p}^{p}+\ldots+\lVert f^{(k)}\rVert_{p}^{p}\right)^{1/p}, \quad f\in C_0^{\infty}(\R).
\end{equation}
The elements of $W^{k,p}$ are so-called fundamental sequences $\{f_j\}_{j\geq0}$, that is, Cauchy sequences in the above norm where two sequences $\{f_j\}$ and $\{g_j\}$ are considered equal if 
$$
\nl f_j-g_j\nr_{W^{k,p}}\to0,\quad j\to\infty.
$$
We define two canonical mappings on $W^{k,p}$. 
Note that if $\{f_j\}\in W^{k,p}$ then $\{f_j\}$ is Cauchy in $L^p$-norm, and $\{f_j'\}$ is Cauchy in $W^{k-1,p}$. 
For $f=\{f_j\}\in W^{k,p}$ we can thus define mappings by $\alpha f=\lim f_j$ and $\delta f=\{f_j'\}$, where the limit involved is taken in $L^p$.

We will at times find it necessary to use indices to distinguish between the mappings \mbox{$\beta=\beta_k: L^{p}\to W^{k,p}$} for different $k$. 
Similarily we denote $\delta$ by $\delta_k$ when necessary.
 
As is the case for $p\geq1$, the mappings $\alpha$ and $\delta$ are both linear and continuous. 
One would perhaps expect that $\alpha$ is also injective, but as was originally shown by Douady who constructed a counter-example, this is not the case. 
There does not seem to be any good original reference, but the Douady example is discussed in detail in \cite{Peetre}. The example says that if the elements of $W^{k,p}$ could be seen as honest to God functions, there would exist a function which is zero but with derivative equal to one!

\section{Addendum to Peetre's proof}\label{correction}

We begin by a result ensuring that we can construct a continuous retraction \mbox{$\beta:L^p\hookrightarrow W^{k,p}$} of the mapping $\alpha$. 
This is due to Peetre and is explained in his paper \cite{Peetre}.

\begin{lemma}\label{retraction} 
There exist a retraction $\beta:L^p\hookrightarrow W^{k,p}$ that is linear, continuous, injective and satisfies 
$$
\alpha \circ \beta = \id \and \delta \circ \beta = 0.
$$
\end{lemma}

We will need the following lemma (see \cite[Lemma~2.1]{Peetre}). 
We recall that the space $C^{k}_{0,+}(I)$ is defined as the space of $k$ times piecewise differentiable functions with compact support; i.e.~$f\in C^{k}_{0,+}$ iff $f$ is compactly supported, has $k-1$ continuous derivatives and its $k$:th derivative, taken in the distributional sense, is piecewise continuous with finitely many discontinuities. A typical class of examples of such functions are the so-called spline functions from approximation theory, see e.g.~\cite{Ahlberg} for a treatment of these.

\begin{lemma}\label{splinedense}
The spaces $C_0^{\infty}$ and $C^{k}_{0,+}$ have the same completion $W^{k,p}$ in the quasi-norm \eqref{sobolevnorm}.
\end{lemma}

The next step is to use the retraction $\beta$ from Lemma~\ref{retraction} to deduce a few facts regarding the mapping $\delta$. The most important is surjectivity. The proof of this fact is perhaps to be thought of as the main result in this note. To get the rather simple underlying idea across, we summarize what is about to happen. 

We consider an element $g=\{g_j\}\in W^{k-1,p}$ and want to find an element $f\in W^{k,p}$ so that $\delta f=g$. In essence, what we try to do is to find a primitive $f_j$ to each $g_j$ and consider $\{f_j\}$ as a $W^{k,p}$-function. To do this we are tempted to integrate. However, to make sure that the resulting function lies in $L^p$, we will need to modify $g_j$ so that it has mean zero. As soon as that has been accomplished, we consider primitives $u_j$. They will lie in $L^p$, but the second obstacle is that $\{u_j\}$ may very well diverge. Using the retraction $\beta$, we can find $L^p$-functions $\tilde{u}_j$ which approximate $u_j$ in $L^p$ and set $f_j=u_j-\tilde{u}_j$ so that convergence is no issue. Since $\delta\circ\beta=0$ this is accomplished without changing too much in the derivatives. 

If one for a moment considers $W^{k,p}$ as a direct sum, where the $n$:th component is the $n$:th derivative, one can view $f=\{f_j\}$ as $(0,g,g',\ldots, g^{(k-1)})$. Thus we actually find a retraction $\gamma$ of $\delta$ which will furthermore satisfy $\alpha\circ\gamma=0$.

\begin{lemma}\label{deltalemma}
The mapping $\delta$ is surjective as a mapping $W^{k,p}\to W^{k-1,p}$ for $k\geq1$ and $\ker \delta=\im \beta$. 
\end{lemma}

\begin{proof}
We begin reproducing Peetre's proof that $\ker \delta=\im \beta$.
One direction is obvious since $\delta\circ\beta=0$. For the other, note that $\delta$ is injective on $\ker \alpha$. 
Indeed, we can calculate the norm
$$
\nl f\nr_{W^{k,p}}^{p}=\nl \alpha f\nr_{p}^{p}+\nl \delta f\nr_{W^{k-1,p}}^{p}.
$$
This gives that for $\delta f$ to be $0$ if $\alpha f=0$, we will need $\nl f\nr_{W^{k,p}}$ to be zero. 
Thus $f=0$. 
Now let $f\in \ker \delta$ and put $f=f_1+f_2$ with $f_1=\beta\circ \alpha f$. 
Clearily $f_1$ lies in $\im \beta$ and $f_2$ in $\ker \alpha$. 
Thus $\delta f=\delta f_2$. However, $\delta f=0$ implies $f_2=0$ by injectivity of $\delta_{\vert \ker \alpha}$ and therefore $f=f_1\in \im \beta$.

The second part of the proof is, as mentioned, trickier. First we note that a general $g\in W^{k-1,p}$ can be seen as a Cauchy sequence $\{g_j\}$ of functions in $C^{k-1}_{0,+}$ by Lemma~\ref{splinedense}. Next we consider another sequence $\{\tilde{g}_j\}$ that is a slight modification that represents the same element in $W^{k-1,p}$ while each each $\tilde{g}_j$ has mean $0$, so that we can define
$$
u_j(x)=\int_{-\infty}^{x}\tilde{g}_j(t)\dd t,\quad j\geq1
$$
without fear of $u_j$ leaving $L^{p}$. 
The idea is then to use the map $\beta$ to lift $u_j$ to a function $\tilde{u}_j$ such that $\{\tilde{u}_j\}$ approximates $u$ but with $\{\tilde{u}'_j\}=0$ in $W^{k-1,p}$. 
Lastly we set $f_j=u_j-\tilde{u}_j$ and set off trying to prove that $f:=\{f_j\}$ is actually an element of $W^{k,p}$ such that $\delta f=g$.

To make this precise we first show that one can construct $\{\tilde{g}_j\}$ as above with mean zero. 

We define $\{\tilde{g}_j\}$ by
$$
\tilde{g}_j=g_j-\left(\int_{\R} g_j(t)\dd t\right)\psi_j,\quad j\geq1
$$
where $\psi_j$  are compactly supported, $\int \psi_j=1$ and satisfy $\nl g_j\nr_1\nl \psi_j\nr_{W^{k-1,p}}\to0$. That such functions $\psi_j$ can be constructed is a typical feature arising when $0<p<1$. Too see why it works, apply for example \cite[Lemma~3.7]{BehmWennman} for $M_n=1$ for $n\geq0$.

To see that $\{\tilde{g}_j\}$ represents the same element in $W^{k-1,p}$ as $\{g_j\}$, we just note that
$$
\nl g_j-\tilde{g}_j\nr_{W^{k-1,p}}=\nl\int_{\R}g_j(t)\dd t\cdot \psi_j\nr_{W^{k-1,p}}\leq \nl g_j\nr_1 \nl \psi_j\nr_{W^{k-1,p}}\to0
$$
by construction.

We move on to define $u=\{u_j\}$ by
$$
u_j(x)=\int_{-\infty}^{x}\tilde{g}_j(t)\dd t,\quad j\geq 1.
$$ 
It is clear that $u_j\in L^p$, so we can use the retraction $\beta$ to obtain a sequence $\{v_{j,n}\}_{n\geq1}$ such that $v_{j,n}\to u_j$ in $L^{p}$ and $v_{j,n}'\to0$ in $W^{k-1,p}$ for each fixed $j$, as $n\to\infty$. 
For each $j$ we let $n(j)$ be large enough for the following to hold
$$
\max\left\{\nl u_j-v_{j,n(j)}\nr_{p}, \nl v_{j,n(j)}'\nr_{W^{k-1,p}}\right\}< j^{-1}.
$$
Put $\tilde{u}_j=v_{j,n(j)}$ and set $f_j=u_j-\tilde{u}_j$. Recall that our goal is to show that $f:=\{f_j\}_{j\geq1}\in W^{k,p}$ and that $\delta f=g$. 

That $f$ lies in $W^{k,p}$ follows from the identity
$$
\nl f_i-f_j\nr_{W^{k,p}}^{p}=\nl f_i-f_j\nr_{p}^{p}+\nl f_{i}'-f_{j}'\nr_{W^{k-1,p}}^{p}.
$$
Indeed, the first part tends to zero since already $f_j$ tends to zero $L^{p}$ by construction (if $\gamma g:=f$, this shows that $\alpha \circ \gamma (g)=0$). The second part can be handled since
$$
\nl f_i'-f_j'\nr_{W^{k-1,p}}^{p}\leq \nl \tilde{g}_i-\tilde{g}_j\nr_{W^{k-1,p}}^{p}+\nl \tilde{u}_i'\nr_{W^{k-1,p}}^{p}+\nl \tilde{u}_j' \nr_{W^{k-1,p}}^{p}.
$$
The first term tends to zero since $\{\tilde{g}_j\}$ is an element of $W^{k-1,p}$, and the last two terms tends to zero since $\nl \tilde{u}_{j}'\nr_{W^{k-1,p}}<j^{-1}$.

It is now within reach to see that $\delta f=g$. 
Indeed, we could take $\{\tilde{g}_j\}$ as a representative of $g$, and
$$
\nl \tilde{g}_j-f_j'\nr_{W^{k-1,p}}^{p}=\nl \tilde{g}_j-\tilde{g}_j+\tilde{u}_{j}'\nr_{W^{k-1,p}}^{p}\to0, \quad j\to\infty
$$
so $\delta f=\{f_j'\}=\{\tilde{g}_j\}=g$.
\end{proof}

\begin{remark}
The above proof works for $W^{k,p}$ for all $k\geq1$. 
In the case $k=1$ we can realize the target space $W^{0,p}$ of $\delta$ as a space of functions without any difficulty; it is actually $L^{p}$. 
With this knowledge, we single out dense set $\Sigma\subset L^{p}$ and construct for each $g\in \Sigma$ an element $f\in W^{1,p}$ which maps to $g$. 
This should be possible to do rather easily when $\Sigma$ denotes the set of step functions by a method very similar to that used to construct the Douady counter-example by using skewed saw tooth-like functions which have very small quasi-norm while having derivative $1$ on a large portion of an interval.

It turns out that it suffices to prove that $\delta$ has dense range. 
Indeed, we know from general theory that an operator which is bounded from below must have closed range, see Proposition~\ref{range} below. 
Note that it is enough to consider $\delta$ on $\ker \alpha$. This follows since $\delta f =\delta f-\delta \beta\circ\alpha f$ by the fact that $\delta\circ\beta=0$ and therefore if $f\notin \ker\alpha$ maps to $g$, so will $f-\beta\circ\alpha f$ which lies in $\ker\alpha$. 

Now, $\delta_{\vert \ker \alpha}$ is bounded from below, so it would follow that if it has dense range, it must be a surjection. 
To see why it is bounded from below, let $f\in \ker \alpha$. 
Then
$$
\nl \delta f\nr_{W^{0,p}}^{p}=\nl \alpha f\nr_{p}^{p}+\nl \delta f\nr_{W^{0,p}}^{p}=\nl f\nr_{W^{1,p}}.
$$

We have actually cheated a bit. 
We appeal to the general theory of Banach spaces when saying that if an operator has dense range and is bounded from below, then it is surjective, but our setting is that of quasi-Banach spaces.
However, the result is exactly the same in the case at hand.

\begin{proposition}\label{range} 
Suppose $X$ and $Y$ are quasi-Banach spaces and $T:X\to Y$ is a bounded operator. 
Then if $T$ is bounded from below, the range of $T$ is closed.  
\end{proposition}

\begin{proof}
Suppose that $T$ is bounded from below. 
To see that it has closed range, suppose that $Tx_n\to y$ in the range $R(T)\subset Y$. 
But then we get
$$
\nl Tx_n-Tx_m\nr_Y\geq C \nl x_n-x_m\nr_X,
$$
for some constant $C$, so $x_n$ tends to some element $x\in X$. 
By boundedness of $T$, we arrive at 
\begin{equation*}
Tx=\lim Tx_n = y. \qedhere
\end{equation*}
\end{proof}
\end{remark}

We are finally able to give the last details of the proof of Theorem~\ref{Peetre-main}.

\begin{proof}[Proof of Theorem~\ref{Peetre-main}]
We shall explicitly give the isomorphism. 
It will be given by
$$
f\mapsto (\alpha f, \delta f),\quad f\in W^{k,p}.
$$
It is readily verified that it is continuous, linear and maps $W^{k,p}$ into the space $L^p \oplus W^{k-1,p}$. 

Injectivity poses no difficulty. 
Indeed, suppose that $f$ maps to $0$. In particular this implies that $\delta f=0$, so there exists a $g$ with $\beta g=f$ since $\ker \delta=\im \beta$. 
But then 
$$
0=\alpha f= \alpha\circ  \beta g = g,
$$
so $f=\beta g=0$.

To finish the proof, we need to show that the mapping defined above is also surjective. 
We thus pick a pair $(g,h)\in L^{p}\oplus W^{k-1,p}$ and try to find an $f\in W^{k,p}$ that maps to this element. 
We do this in two steps. 
First, choose a function $f_0$ such that $\delta f_0=h$. Then set $f=f_0-\beta(\alpha f_0-g)$. 
Since $\delta\circ \beta=0$ it follows that 
$$
\delta f=\delta f_0-\delta\circ\beta( \alpha f_0- g) =\delta f_0=h.
$$ 
Now, since $\alpha\circ\beta=\id$ we get
$$
\alpha f=\alpha(f_0-\beta \alpha f_0+\beta g)=f_0-\alpha\circ \beta (\alpha f_0- g)=\alpha (f_0-f_0)+g=g,
$$
and the proof is complete.
\end{proof}

%\section{The Douady Counter-example for $k>1$}
%As a last remark we use the machinery already developed to construct analogues to the Douady Conterexample \cite[]{Peetre} for $k> 1$. That is, we illustrate how one can convince oneself that $\alpha$ is not injective.
%
%We will do this by showing that there are a multitude of elements $f\in W^{k,p}$ such that $\alpha f=0$.
\bibliography{References}{}
\bibliographystyle{amsplain}

\end{document}